\documentstyle[11pt,amssymb,amsmath,amscd,amsbsy,amsfonts,theorem,hyperref]{article}

\input xy
\xyoption{all}

\newcommand{\zz}{{\Bbb Z}}

\newcommand{\nn}{\Bbb N}
\newcommand{\cc}{\Bbb C}
\newcommand{\rr}{\Bbb R}

\newcommand{\pp}{{\Bbb P}}
\newcommand{\aaa}{{\Bbb A}}

\newcommand{\hh}{{\Bbb H}}
\newcommand{\js}{{\frak{j}}}
\newcommand{\jjs}{{\mathbf{j}}}

\newcommand{\ddim}{\operatorname{dim}}

\newcommand{\ddetpm}{\operatorname{det}_{\pm}}

\newcommand{\Hom}{\operatorname{Hom}}
\newcommand{\op}[1]{\operatorname{#1}}
\newcommand{\kbar}{\overline{k}}

\newcommand{\ffi}{\varphi}

\newcommand{\la}{\langle}
\newcommand{\ra}{\rangle}
\newcommand{\lva}{\langle\!\langle}   
\newcommand{\rva}{\rangle\!\rangle}   
\newcommand{\row}{\rightarrow}

\newcommand{\low}{\leftarrow}
\newcommand{\lrow}{\longrightarrow}
\renewcommand{\leq}{\leqslant}

\newcommand{\nichego}[1]{}

\newcommand{\ov}[1]{\overline{#1}}

\newcommand{\wt}[1]{\widetilde{#1}}

\newcommand{\smk}{{\mathbf{Sm_k}}}

\newcommand{\cm}{{\cal M}}

\newcommand{\dmk}{\op{DM}(k)}
\newcommand{\dmkD}{\op{DM}(k;\zz/2)}
\newcommand{\hii}{{\cal X}}
\newcommand{\DQMgm}{\op{DQM}^{gm}}

\newcommand{\Qed}{\hfill$\square$\smallskip}
\newenvironment{proof}{\noindent{\it Proof}:}{\vskip 5mm}
\newenvironment{prof}[1]{\noindent{\it #1}:}{\vskip 5mm}

\newenvironment{remarks}{\noindent{\bf Remarks:}}{\vskip 5mm}

\newtheorem{prop}{Proposition}[section]{\bf}{\it}
\newtheorem{thm}[prop]{Theorem}{\bf}{\it}
\newtheorem{lem}[prop]{Lemma}{\bf}{\it}
{\bf}{\it}
{\bf}{\it}
{\bf}{\it}
{\bf}{\it}
{\bf}{\it}
{\bf}{}

{\bf}{\it}
{\bf}{\it}
{\bf}{\it}
{\bf}{\it}

\begin{document}

\title{Motivic equivalence of affine quadrics}
\author{Tom Bachmann\footnote{Fakult\"at Mathematik, Universit\"at Duisburg-Essen} {\tiny and} Alexander Vishik\footnote{School of Mathematical Sciences, University
of Nottingham}}
\date{}

\maketitle

\begin{abstract}
In this article we show that the motive of an affine quadric $\{q=1\}$ determines the respective quadratic form.
\end{abstract}

\tableofcontents

\section{Introduction}
\label{Intro}
The algebraic theory of quadratic forms over arbitrary fields was created by E.Witt in the 1930-ies. It was greatly extended in the 60-ies
by A.Pfister and then by J.Milnor, J.-Kr.Arason, M.Knebusch and others. Aside from the classical algebraic methods, various
cohomological tools and specialization arguments entered the game. For a while the subject was centered around Milnor's conjecture
and the forms which were attracting the main attention were Pfister forms, since the respective projective quadrics are norm-varieties
for pure symbols in Milnor's K-theory mod 2. One should note here the contribution by A.Merkurjev, A.Suslin and M.Rost.
Simultaneously, unramified cohomology was successfully applied by B.Kahn, M.Rost, R.Sujatha and O.Izhboldin.
In the works of M.Rost, motives entered the subject with the discovery of the Rost motive and the proof of such
foundational results as the Nilpotence Theorem - see \cite{R,R2}. The Pfister case was developed further by V.Voevodsky in the context of
triangulated category of motives in \cite{VoMil} and investigated in more details in \cite{OVV}. Finally, the case of an arbitrary quadric
was systematically addressed from the motivic point of view in the works of the second author, and of N.Karpenko and A.Merkurjev.
This approach is based on studying geometric and motivic properties of various projective homogeneous varieties associated to a quadric,
and it permitted to settle many open problems left from the algebraic theory. For an overview see \cite{lens} and \cite{EKM}
(and also \cite{EC}).
Note however that in these works only the motives of projective quadrics were studied, while the affine case was largely ignored.

The story made a new turn when F.Morel showed that the Grothendieck-Witt ring of quadratic forms over $k$ can be identified with
the $(0)[0]$-stable homotopy group of spheres in the algebro-geometric (motivic) homotopic world \cite[Theorem 6.2.1]{morel2004motivic-pi0}. This reignited the interest in quadratic
forms due to the prominent role these objects play in $\aaa^1$-homotopy theory and, in particular, due to
understanding that their properties should be related to that of the classical (topological) stable homotopy groups of spheres.

From a homotopic point of view, an affine quadric is more attractive than the projective one. Over an algebraically closed field, such an affine quadric has only two cells, like a sphere. Moreover over $\cc$ the complex points of an affine quadric have the homotopy type of a topological sphere, and the same holds for the real points over $\rr$. Over a general field, one may consider affine quadrics to be ``non-split'' spheres. It appears that their motives behave much better than those of projective quadrics.
While the motive of a projective quadric does not determine the quadric itself, the motive of an affine quadric determines the
respective form.
In this article we will show that the motives of two affine quadrics $\{q=1\}$ and $\{p=1\}$ are isomorphic if and only if $p\cong q$.
Thus, we obtain that the restriction of the functor
$\smk\row {\cal H}(k)\row S{\cal H}(k)\row \dmk$ to affine quadrics does not glue objects.

We provide two alternative proofs of this result. One uses the standard motivic techniques developed for projective quadrics (and projective
homogeneous varieties, in general). The second one uses the novel "generalised geometric fixed point functors'' introduced
by the first author in \cite{BQ}. We hope that this will provide a good illustration of the methods of \cite{BQ}.
In particular, the reader should see how these methods interact with the classical techniques. Both versions use the Theorem of Izhboldin on
motives of odd-dimensional projective quadrics \cite{Izh}.

\section{The Main Theorem}

Everywhere below $k$ will be a field of characteristic different from $2$.
For a quadratic form $q$ over $k$ we will denote by $A_q$ the affine quadric $\{q=1\}$.
For a smooth variety $X$ over $k$, we denote by $M(X)$ its image in the triangulated categories of motives $\dmk$ or
$\dmkD$ of V.Voevodsky (constructed in \cite{voevodsky-triang-motives}; we drop "minus" from the notations).

\begin{thm}
\label{Main}
Let $p$ and $q$ be quadratic forms over $k$, and $A_q$ and $A_p$ be the respective affine quadrics.
Then the following conditions are equivalent:
\begin{itemize}
\item[$(1)$] $p\cong q$;
\item[$(2)$] $M(A_q)\cong M(A_p)$ in $\dmk$;
\item[$(3)$] $M(A_q)\cong M(A_p)$ in $\dmkD$.
\end{itemize}
\end{thm}

\begin{proof}
It is obvious that $(1)\Rightarrow(2)\Rightarrow(3)$. To prove that $(3)\Rightarrow(1)$ we will need the following result.

For quadratic forms $q$ and $p$ over $k$, denote by $Q$ (respectively $P$) the projective quadric given by $q$ (respectively, $p$),
and by $Q'$ (respectively $P'$) the projective quadric given by $q\perp\la -1\ra$ (respectively, $p\perp\la -1\ra$).

\begin{prop} \label{prop:equiv-1}
\label{QQ'}
Suppose that $M(A_q)\cong M(A_p)$ in $\dmkD$. Then $M(Q)\cong M(P)$ and $M(Q')\cong M(P')$ (in the same category).
\end{prop}

We will prove this proposition by emplyoing the criterion of motivic equivalence of quadratic forms. To state it, recall that if $q$ is a (non-degenerate) quadratic form over a field $k$, then $q \cong n\hh \perp q_{an}$, where $q_{an}$ is anisotropic. Moreover, $n$ and (the isomorphism class of) $q_{an}$ are determined uniquely \cite{L}. The number $n$ is called the \emph{Witt index} of $q$ over $k$ and we denote it by $i_W(q)$. The criterion of motivic equivalence now states that for non-degenerate quadratic forms $p, q$ we have $M(Q) \cong M(P)$ if and only if $dim(p) = dim(q)$ and $i_W(p|_E) = i_W(q|_E)$ for all field extensions $E/k$ - \cite[Proposition 5.1]{IMQ} or \cite[Theorem 4.18]{lens}
(see also \cite{Kar2}). Thus Proposition \ref{prop:equiv-1} is equivalent to the following result:

\begin{prop} \label{prop:equiv-2}
Suppose that $M(A_q)\cong M(A_p)$ in $\dmkD$. Then $\dim(p) = \dim(q)$, $\dim(p') = \dim(q')$, and for all field extensions $E/k$ we have $i_W(p|_E) = i_W(q|_E)$ and $i_W(p'|_E) = i_W(q'|_E)$.
\end{prop}

We give here a direct proof using motivic methods developed for projective quadrics.
We still introduce some new objects: the {\it splitting tower} and {\it shells} for an affine quadric. These are our main tools.

\begin{proof}
Since $A_q$ is an open subvariety of $Q'$ with closed (regular) complement $Q$, we have a distinguished Gysin triangle
\[ M(Q')\row M(Q)(1)[2]\row M(A_q)[1]\row M(Q')[1] \] in $\dmk$, and similarly for $p$. Notice, that the group
$\Hom_{\dmkD}(M(Q)(1)[2], M(P')[1])$ is zero, since, for smooth $X$,
$\op{H}^{b,a}_{\cm}(X)=0$, for $b>2a$ (see e.g. \cite[Corollary 4.2.6]{voevodsky-triang-motives}). Similarly, there are no hom's from $M(P)(1)[2]$ to $M(Q')[1]$.
This shows that our (mutually inverse) isomorphisms
\[ \xymatrix@-0.1pc{M(A_q) \ar @/^0.5pc/ @{->}[rr]^(0.5){f} & & M(A_p) \ar @/^0.5pc/ @{->}[ll]^(0.5){g}} \]
can be extended to morphisms ${\mathbf{\Phi}}:{\mathbf{Q}}\row {\mathbf{P}}$ and
${\mathbf{\Psi}}:{\mathbf{P}}\row {\mathbf{Q}}$ of distinguished
triangles:
$$
\xymatrix @-0.7pc{
M(Q') \ar@{->}[r] \ar@{->}[d]^(0.5){\ffi'} & M(Q)(1)[2] \ar@{->}[r] \ar@{->}[d]^{\ffi} & M(A_q)[1] \ar@{->}[r] \ar@{->}[d]^(0.5){f[1]} &
M(Q')[1] \ar@{->}[d]^{\ffi'[1]}\\
M(P') \ar@{->}[r] \ar@{->}[d]^(0.5){\psi'} & M(P)(1)[2] \ar@{->}[r] \ar@{->}[d]^{\psi} & M(A_p)[1] \ar@{->}[r] \ar@{->}[d]^(0.5){g[1]} &
M(P')[1] \ar@{->}[d]^{\psi'[1]}\\
M(Q') \ar@{->}[r] & M(Q)(1)[2] \ar@{->}[r] & M(A_q)[1] \ar@{->}[r] & M(Q')[1].
}
$$
Let ${\mathbf{A}}:={\mathbf{\Psi}}\circ{\mathbf{\Phi}}$, ${\mathbf{B}}:={\mathbf{\Phi}}\circ{\mathbf{\Psi}}$. Respectively,
$\alpha:=\psi\circ\ffi$, $\beta:=\ffi\circ\psi$, $\alpha':=\psi'\circ\ffi'$, and $\beta':=\ffi'\circ\psi'$.
When restricted to algebraic closure, $\alpha,\beta,\alpha'$ and $\beta'$ are endomorphisms of split motives of finite rank with
$\zz/2$-coefficients. The respective endomorphism rings are finite, and so, some power of every element is an idempotent. We clearly
may choose this power to be common for all four elements in question (just multiply the individual powers). Then it follows from the
Rost Nilpotence Theorem for quadrics - \cite{R2} (see also \cite[Lemma 3.10]{IMQ}) that there is $N\in\nn$ such that
$\alpha^N,\beta^N,(\alpha')^N,(\beta')^N$ are idempotents - see \cite[Corollary 3.2]{lens}. Changing ${\mathbf{\Psi}}$ to
${\mathbf{A}}^{\circ 2N-1}\circ{\mathbf{\Psi}}$
(and keeping ${\mathbf{\Phi}}$), we may assume that $N=1$, i.e. that $\alpha$, $\beta$, $\alpha'$, $\beta'$ are
idempotents. Finally, changing ${\mathbf{\Psi}}$ to ${\mathbf{A}}\circ{\mathbf{\Psi}}\circ{\mathbf{B}}$ and
${\mathbf{\Phi}}$ to ${\mathbf{B}}\circ{\mathbf{\Phi}}\circ{\mathbf{A}}$, we may assume that, in addition,
${\mathbf{\Psi}}\circ{\mathbf{B}}={\mathbf{\Psi}}={\mathbf{A}}\circ{\mathbf{\Psi}}$ and
${\mathbf{\Phi}}\circ{\mathbf{A}}={\mathbf{\Phi}}={\mathbf{B}}\circ{\mathbf{\Phi}}$.

Denoting as $({\mathbf{Q}},{\mathbf{A}})$ the image of the projector ${\mathbf{A}}$ and as
$({\mathbf{Q}},\ov{{\mathbf{A}}})$ its kernel (and similarly for $p$), we see that ${\mathbf{\Phi}}$ is zero on $({\mathbf{Q}},\ov{{\mathbf{A}}})$,
while ${\mathbf{\Psi}}$ is zero on $({\mathbf{P}},\ov{{\mathbf{B}}})$. At the same time, we get isomorphisms
\[
  ({\mathbf{Q}},{\mathbf{A}})\stackrel{{\mathbf{\Phi}}}{\lrow}({\mathbf{P}},{\mathbf{B}})
  \quad\text{and}\quad
   ({\mathbf{P}},{\mathbf{B}})\stackrel{{\mathbf{\Psi}}}{\lrow}({\mathbf{Q}},{\mathbf{A}}).
\]
This gives
(mutually inverse) isomorphisms:
\[
  \xymatrix@-0.1pc{(M(Q'),\alpha') \ar @/^0.5pc/ @{->}[rr]^(0.5){\ffi'} & & (M(P'),\beta') \ar @/^0.5pc/ @{->}[ll]^(0.5){\psi'}}
  \quad\text{and}\quad
 \xymatrix@-0.1pc{(M(Q)(1)[2],\alpha) \ar @/^0.5pc/ @{->}[rr]^(0.5){\ffi} & & (M(P)(1)[2],\beta) \ar @/^0.5pc/ @{->}[ll]^(0.5){\psi}}.
\]

On the other hand, the affine components of $({\mathbf{Q}},\ov{{\mathbf{A}}})$ and $({\mathbf{P}},\ov{{\mathbf{B}}})$ are zero
(since $g$ and $f$ are mutually inverse).
This provides isomorphisms $(M(Q'),\ov{\alpha'})\stackrel{\cong}{\lrow} (M(Q)(1)[2],\ov{\alpha})$ and
$(M(P'),\ov{\beta'})\stackrel{\cong}{\lrow} (M(P)(1)[2],\ov{\beta})$.

Recall that over $\kbar$ the motive of a projective quadric $R$ of dimension $m$ (corresponding to a form $r$) splits into a direct sum
$\oplus_{0\leq i\leq d}(\zz/2(i)[2i]\oplus\zz/2(m-i)[2m-2i])$
of Tate motives (here $d=[m/2]$). Let us denote the respective Tate motives as $T_i$ and $T^i$. For a field extension $E/k$
the Tate motives $T_i$ and $T^i$ split from $M(R)$ if and only if $i_W(r_E)>i$. We say that $i$ and $j$ belong to the same {\it shell}
if the conditions $i_W(r_E)>i$ and $i_W(r_E)>j$ are equivalent (for all $E/k$). The latter condition is equivalent to the fact that
there are rational
maps \[\xymatrix@-0.1pc{R_i \ar @/^0.5pc/ @{-->}[rr]^(0.5){} & & R_j \ar @/^0.5pc/ @{-->}[ll]^(0.5){}},\] where $R_l$ is the Grassmannian of $l$-dimensional projective subspaces on $R$. By the Theorem of Springer (see \cite[VII,Theorem 2.3]{L}) this is equivalent to the fact that
the varieties $R_i|_{k(R_j)}$ and $R_j|_{k(R_i)}$ have zero-cycles of degree one. And due to \cite[Theorem 2.3.4]{IMQ} we can reformulate
it as: the motives $M(\hii_{R_i})$ and $M(\hii_{R_j})$ of the Chech simplicial schemes are isomorphic (such an isomorphism is then unique).
Recall, that if $V$ and $W$ are smooth varieties such that $W_{k(V)}$ has a zero cycle of degree $1$,
then there is a natural map $M(\hii_V)\row M(\hii_W)$ (a unique such non-zero
map in $\dmkD$, it becomes an isomorphism of Tate-motives over $\kbar$) - see, for example,
\cite[Theorem 2.3.6]{IMQ}. Since we always have rational maps
$\xymatrix@-0.1pc{R_i & R_j \ar @{-->}[l]}$, for $i<j$, we get a chain of morphisms:
$$
M(\hii_{R_0})\low M(\hii_{R_1})\low M(\hii_{R_2})\low M(\hii_{R_3})\low\ldots.
$$
And shells are just maximal connected pieces consisting of isomorphisms in this chain of morphisms.

It follows from the work of M.Knebusch \cite{Kn2} that the notion of {\it shells} is equivalent to that of a
{\it splitting pattern} $\jjs(r)$ of the form $r$ which is defined as an increasing sequence
$\{\js_0,\js_1,\ldots,\js_h\}$ of all possible Witt indices of $r_E$ over
all possible field extensions $E/k$. Then $i$ and $l$ belong to the same shell if and only if $\js_t\leq i,l<\js_{t+1}$, for some $t$.

\begin{lem}
\label{chech-eq}
For any $0\leq i\leq [\ddim(Q')/2]$ we have
$$
\xymatrix@-0.1pc{
\ar@{}[rd]|-{\text{either:}}& & M(\hii_{Q'_i}) \ar@{=}[d] \ar@{}[rd]|-{\text{and}} & M(\hii_{Q_i})\ar@{=}[d] \ar@{}[rrd]|-{or:} &  &
 M(\hii_{Q'_i}) \ar@{=}[r] \ar@{}[rd]|-{\text{and}} & M(\hii_{Q_i}) \ar@{}[rd]|-{.}&\\
& & M(\hii_{P'_i}) & M(\hii_{P_i}) & & M(\hii_{P'_i}) \ar@{=}[r] & M(\hii_{P_i}) &
}
$$
\end{lem}

\begin{proof}
Denote $(M(Q'),\alpha')\cong (M(P'),\beta')$ as $N'$. Then, for any $i$ such that $T^i$ belongs to the decomposition of $N'_{\kbar}$, the
conditions that $T^i$ splits from $M(Q')_E$ and $M(P')_E$ are equivalent, for every $E/k$. This means that the conditions
$i_W(q'_E)>i$ and $i_W(p'_E)>i$ are equivalent as well, and so $M(\hii_{Q'_i})\cong M(\hii_{P'_i})$.  Analogously, denoting
$(M(Q)(1)[2],\alpha)\cong(M(P)(1)[2],\beta)$ as $N$, we see that, for any $i$ such that $T^i$ belongs to the decomposition of
$N_{\kbar}$, the conditions $i_W(q_E)>i$ and $i_W(p_E)>i$ are equivalent, for all $E/k$, and so, $M(\hii_{Q_i})\cong M(\hii_{P_i})$.

If, on the other hand,  $T^i$ is not contained in $N'_{\kbar}$, then it is contained in $(M(Q'),\ov{\alpha'})_{\kbar}$.
But $(M(Q'),\ov{\alpha'})$
is identified with $(M(Q)(1)[2],\ov{\alpha})$. This implies that the conditions $i_W(q'_E)>i$ and $i_W(q_E)>i$ are equivalent,
for every $E/k$. In other words, $M(\hii_{Q'_i})\cong M(\hii_{Q_i})$.
Analogously, we also have: $M(\hii_{P'_i})\cong M(\hii_{P_i})$.
\Qed
\end{proof}

Since $q$ is a subform of codimension one in $q'$, for any $i$, we have
implications:
$$
i_W(q_E)>i+1\,\,\Rightarrow\,\,i_W(q'_E)>i+1\,\,\Rightarrow\,\,i_W(q_E)>i\,\,\Rightarrow\,\,i_W(q'_E)>i.
$$
In particular, in analogy with the classical {\it splitting tower of Knebusch} - see \cite{Kn2},
it makes sense to speak about the {\it common splitting tower} for $Q$ and $Q'$ where the splitting fields
for $Q$ and $Q'$ are intertwined. More precisely, we may introduce the partial order on finitely generated extensions
over $k$, by saying that $K/k<L/k$ iff $K/k$ can be embedded into some purely transcendental extension of $L/k$. We say that
$K/k\sim L/k$ if $K/k>L/k$ and $L/k>K/k$.
The above implications give us the chain of rational maps
\[\xymatrix@-0.1pc{Q'_0 & Q_0\ar @{-->}[l] & Q'_1\ar @{-->}[l] & Q_1\ar @{-->}[l] & \ldots \ar @{-->}[l]   }.\]
And quadratic Grassmannians are rational as soon as they posses a rational point.
This gives the chain of fields (where we drop the ground field from the notations):
$$
k(Q'_0)<k(Q_0)<k(Q'_1)<k(Q_1)<k(Q'_2)<k(Q_2)<\ldots.
$$
By substituting the field extensions by equivalent ones, we can change "inequalities" to embeddings of fields,
thus creating a genuine tower.
We will call such tower the {\it splitting tower of $A_q$}, or just an {\it affine splitting tower}.
In a different language, we have an ordered chain of motives of Chech simplicial schemes:
$$
M(\hii_{Q'_0})\low M(\hii_{Q_0})\low M(\hii_{Q'_1})\low M(\hii_{Q_1})\low\ldots.
$$
Let us call a maximal connected piece consisting of isomorphisms in this chain - an "affine shell" of the pair $(Q',Q)$
(or, in other words, an "affine shell of $A_q$").
And since the distance between the anisotropic parts of $q'$ and $q$ is always one,
for any extension $L/k$, it follows from \cite[Corollary 4.9]{lens} (or \cite[Corollary 3]{DSMQ}) that
$i_1(q'_L)-i_1(q_L)=\ddim((q'_L)_{an})-\ddim((q_L)_{an})=\pm 1$ (here $i_1(r)=\js_1(r)-\js_0(r)$ is the {\it $1$-st higher Witt index}).
That means that every positive "affine shell" of $(Q',Q)$ ends with the same quadric with which it starts.

We need to show that, over any field $E/k$, the Witt indices of $q'_E$ and $p'_E$ (respectively, $q_E$ and $p_E$) coincide.
In other words, that the boundary between the 0-th and the 1-st "affine shells" for the pair $(Q',Q)$ is the same as that
for the pair $(P',P)$. Assume our field to be $E$.
There are 2 cases: either the $0$-th affine shell ends with $Q'$ and the $1$-st one starts with $Q$ (we call this a $Q'-Q$-boundary),
or the other way around (the $Q-Q'$-boundary). In the $Q'-Q$-case, for the first element $i=\js_0(q)$ of the 1-st affine shell, we have $M(\hii_{Q'_i})\cong T\not\cong M(\hii_{Q_i})$.
Hence, $T^i$ belongs to $N'_{\kbar}$ and $N_{\kbar}$, and so $M(\hii_{Q'_i})=M(\hii_{P'_i})$ and $M(\hii_{Q_i})=M(\hii_{P_i})$.
So, $M(\hii_{P'_i})\cong T\not\cong M(\hii_{P_i})$, and thus, the boundary between the 0-th and 1-st shell for the pair $(P',P)$ is
in the same place.
Consider now the case of a $Q-Q'$-boundary. Then the 1-st affine shell of $(Q',Q)$ starts and ends with $Q'$. Then for the last element
$i=\js_1(q')-1$ of the 1-st shell, $M(\hii_{Q'_i})\not\cong M(\hii_{Q_i})$. Hence, $T^i$ belongs to $N'_{\kbar}$ and $N_{\kbar}$, and so,
$M(\hii_{Q'_i})\cong M(\hii_{P'_i})$ and $M(\hii_{Q_i})\cong M(\hii_{P_i})$. In particular, the boundary between the 1-st and 2-nd shell
of $(Q',Q)$ is also a boundary between some affine shells for $(P',P)$. Passing to the appropriate splitting field of $(P',P)$ and
using the fact that the fields $k(Q_i)$ and $k(P_i)$ are equivalent, we may assume that the respective affine shell of $(P',P)$
is the 1-st one. Since the 1-st shells for $(Q',Q)$ and $(P',P)$ end in the same place and $M(\hii_{Q'_i})\cong M(\hii_{P'_i})$,
by \cite[Corollary 4.9]{lens}, they
should start in the same place as well. Consequently, our original $Q-Q'$-boundary is also a boundary (between some shells) for $(P',P)$.
Repeating this argument with the pair $(P',P)$, we obtain that the boundaries between the 0-th and 1-st shells for $(Q',Q)$ and
$(P',P)$ coincide. In other words, $i_W(q'_E)=i_W(p'_E)$ and $i_W(q_E)=i_W(p_E)$.

Proposition is proven.
\Qed
\end{proof}

Return to the proof of the Theorem. From Proposition \ref{QQ'} we know that $M(Q')\cong M(P')$, while $M(Q)\cong M(P)$.
Let us show that this implies that $q\cong p$.
Notice that either $Q$, or $Q'$ is odd-dimensional. By changing $q$ to $q'$ and $q'$ to $q\perp\hh$ (and similar for $p,p'$),
if necessary, we can assume that $\ddim(q)$ is odd (and $q'=q\perp\la \pm 1\ra$, $p'=p\perp\la \pm 1\ra$). Indeed by the Witt cancellation theorem, if $p \perp \hh \cong q \perp \hh$ then $p \cong q$.
Then, by the Theorem of Izhboldin - \cite[Theorem 2.5]{Izh}, there exists $\lambda\in k^*$ such that $p\cong \lambda\cdot q$.
In addition to this we have an isomorphism $M(Q')\cong M(P')$. In particular, the quadratic forms $q'=q\perp\la\pm 1\ra$ and
$\lambda\cdot p'=q\perp\la\pm\lambda\ra$ are split simultaneously. Let $k_h$ be the last field in the splitting tower of Knebusch for $q'$.
Then over ${k_h}$ the form $q'$ is split, and hence, so is $\lva\lambda\rva=\pm (q'-\lambda p')$.
Since the algebraic closure of $k$ in $k_h$ is
$k(\sqrt{\ddetpm(q')})$, it follows that either $\lambda=1$ or $\lambda=\ddetpm(q')\in k^*/(k^*)^2$.
And swapping $p,p'$ with $q,q'$, we get that either $\lambda=1$ or $\lambda=\ddetpm(p')\in k^*/(k^*)^2$.
Since $\ddetpm(q')\cdot\ddetpm(p')=\lambda$, we obtain that $\lambda=1$, and so, $q\cong p$.
This proves the implication $(3)\Rightarrow(1)$ and the Theorem.
\Qed
\end{proof}

\section{Alternative Proof of Proposition \ref{prop:equiv-2}}

We recall the ``generalised geometric fixed point functors'' introduced in \cite{BQ}. Let $\DQMgm(k, \zz/2)$ denote the thick
tensor triangulated subcategory of $\dmkD$ generated by $M(X)(i)$ for $X$ a smooth projective quadric over $k$ and $i \in \zz$. Let us note right away that $M(A_q) \in \DQMgm(k, \zz/2)$; this is an immediate consequence of the Gysin triangle \cite[Lemma 28]{BQ}.

We write $Tate(\zz/2)$ for the subcategory of $\dmkD$ consisting of pure Tate motives, i.e. sums of objects of the form $\zz/2(i)[2i]$ with $i \in \zz$. This category is equivalent to the category of finite-dimensional graded $\zz/2$-vector spaces. We denote the homotopy category of bounded chain complexes in $Tate(\zz/2)$ by $K^b(Tate(\zz/2))$; this is equivalent to the category of finite-dimensional bi-graded $\zz/2$-vector spaces.

The following is one of the main results of \cite{BQ}:

\begin{thm} \label{thm:Phi}
There exists an essentially unique tensor triangulated functor
$$
\Phi^k: \DQMgm(k, \zz/2) \to K^b(Tate(\zz/2))
$$
such that:
\begin{enumerate}
\item If $X$ is an anisotropic projective quadric (i.e. $X$ does not have a rational point), then $\Phi^k(M(X)) = 0$.
\item We have $\Phi^k(\zz/2(i)) = \zz/2(i)$.
\end{enumerate}
\end{thm}

\begin{remarks}
\begin{enumerate}
\item If $E/k$ is a field extension, we obtain by composition a functor $\DQMgm(k, \zz/2) \to \DQMgm(E, \zz/2) \xrightarrow{\Phi^E} K^b(Tate(\zz/2))$. We abusively also denote it by
$$
\Phi^E: \DQMgm(k, \zz/2) \to K^b(Tate(\zz/2)).
$$

\item In order to use this result effectively, the following lemma of Rost is useful: if $q = \hh \perp \wt{q}$, then
\cite[Proposition 1]{R} (or \cite[Lemma 28]{BQ})
\begin{equation} \label{eq:rost}
  M(Q) \cong \zz/2 \oplus M(\wt{Q})(1)[2] \oplus \zz/2(n)[2n], \text{ where } n=\dim(Q). \tag{$*$}
\end{equation}
\end{enumerate}
\end{remarks}
As a warmup exercise, let us re-prove the easy half of the criterion for motivic equivalence:

\begin{prop}
Let $P, Q$ be smooth projective quadrics such that $M(P) \cong M(Q) \in \DQMgm(k, \zz/2)$. Then $\dim(P) = \dim(Q)$ and $i_W(P|_E) = i_W(Q|_E)$ for all field extensions $E/k$.
\end{prop}
\begin{proof}
If $P$ has a rational point over a field extension $E$, then \eqref{eq:rost} together with Theorem \ref{thm:Phi} implies that $\dim(P)$ is the maximal weight occurring in $\Phi^E(M(P))$ (in other words the bigraded $\zz/2$-vector space $\Phi^E(M(P))$ is isomorphic to a sum of one-dimensional spaces of the form $\zz/2(i)[j]$, and $\dim(P)$ is the maximal $i$ that occurs). This always happens over an algebraic closure $E=\bar{k}$, so that $M(P) \cong M(Q)$ implies $\dim(P) = \dim(Q)$.

For the Witt index we argue similarly. Let $p = r\hh \perp \wt{p}$ with $\wt{p}$ anisotropic, so that $i_W(P) = r$. Repeated application of \eqref{eq:rost}  shows that\footnote{Note that if $p = aX^2$ is a one-dimensional quadratic form (with $a \ne 0$) then $P \subset \pp^0 = \pp(k^1) = *$ is the closed projective subvariety given by $\{aX^2 = 0\}$, which is \emph{empty}. In particular $M(P)=0$.}
\begin{align*}
  M(P) \cong \oplus_{i=0}^r(\zz/2(i)[2i]\oplus\zz/2(\ddim(P)-i)[2\ddim(P)-2i])\oplus M(\wt{P})(r)[2r]
\end{align*}
Consequently $\Phi^k(M(P))$ has rank $2r$, by Theorem \ref{thm:Phi}. Thus $M(P)$ determines $i_W(P)$, as was to be shown.
By base change to a field extension $E/k$, $M(P)$ also determines $i_W(P|_E)$. This concludes the proof.
\Qed
\end{proof}

In the above proof, we have seen that the bi-graded vector space $\Phi^k(M(P))$ can have large dimension, and we had to work to extract the numbers that we wanted. If instead of a projective quadric we consider an affine quadric, then the vector spaces $\Phi^k(M(A_p))$ become much smaller - they have either dimension 2 or 0. In fact, if instead we consider the reduced motive, then $\Phi^k(\tilde{M}(A_p))$ always has dimension 1: this is expected if $\tilde{M}(A_p)$ is supposed to ``behave like a sphere'', i.e. be invertible, and the main result of \cite{BQ} is that not only is the dimension always 1, but that also this condition is sufficient (and necessary) for invertibility. Consequently if $A_q$ is an affine quadric, then for every field extension $E/k$ we have $\Phi^E(\tilde{M}(A_q))=\zz/2[i_E](j_E)$, and we get two numerical invariants of $q$ (namely $i_E$ and $j_E$). The strategy of the alternative proof of our main result is to show that these invariants are closely related to the Witt indices of $q|_E$ and $q'|_E$.

\begin{prof}{Alternative proof of Proposition \ref{prop:equiv-2}}
Recall that we are given quadratic forms $p$ and $q$ such that $M(A_p) \cong M(A_q)$, and we need to show that $\dim(p) = \dim(q)$, $\dim(p') = \dim(q')$, and for all field extensions $E/k$ we have $i_W(p|_E) = i_W(q|_E)$ and $i_W(p'|_E) = i_W(q'|_E)$.

In the proof we will want to write $p = n\hh \perp p_{an}$, for an anisotropic form $p_{an}$, which we will then treat on a similar footing to $p$. This requires us to allow the degenerate cases $p=0$, where $A_p = P = P' = \emptyset$ and $p=aX^2$, where $A_p = P' = Spec(k[x]/(x^2-a))$ and $P=\emptyset$. We consider both of these forms as anisotropic. Note that it is still the case that $\Phi^k(M(P)) = 0$ in these degenerate cases, and this is the main property we require of anisotropic forms.

Let us first note that $M(A_q)$ determines the reduced motive $\tilde{M}(A_q)$. Indeed we have $\Hom_{\dmkD}(M(A_q), \zz/2) = CH^0(A_q)/2$ \cite[Corollary 4.2.5]{voevodsky-triang-motives}, and the fibre of a\footnote{There is more than one non-zero morphism if and only if $q \cong X^2$, in which case $M(A_q) = M(\{X^2 = 1\}) = M(* \coprod *)$.} non-zero morphism $M(A_q) \to \zz/2$ is $\tilde{M}(A_q)$. If all such morphisms are zero, then $A_q = \emptyset$ and $\tilde{M}(A_q) = \zz/2[-1]$ is also determined.

The following observation will be used many times: if $q \cong \wt{q} \perp \hh$, then it follows from \eqref{eq:rost} - see
\cite[Lemma 34]{BQ} that
\begin{equation} \label{eq:twist}
  \tilde{M}(A_q) \cong \tilde{M}(A_{\wt{q}})(1)[2].
\end{equation}

Let us first show that $\tilde{M}(A_q)$ determines $\dim(Q)$ and $\dim(Q')$. For this we base change to an algebraic closure $\bar{k}$. Then either $q \cong n\hh$ or $q \cong n\hh \perp \la 1 \ra$. In the former case we have $\Phi^{\bar{k}}(\tilde{M}(A_q)) = \Phi^{\bar{k}} (\tilde{M}(\emptyset))(n)[2n] = \zz/2(n)[2n-1]$, whereas in the latter case we have $\Phi^{\bar{k}}(\tilde{M}(A_q)) = \Phi^{\bar{k}} (\tilde{M}(\{x^2 = 1\}))(n)[2n] = \zz/2(n)[2n]$. None of these graded one-dimensional vector spaces are isomorphic, so $\dim(Q)$ can be read off from $\Phi^{\bar{k}}(\tilde{M}(A_q))$. Of course $\dim(Q)$ determines $\dim(Q')$.

Next we show that $\tilde{M}(A_q)$ determines $i_W(q)$ and $i_W(q')$. For this we start with the exact triangle
\[ M(A_q) \to M(Q') \to M(Q)(1)[2] \to M(A_q)[1], \]
from which we deduce the exact triangle
\begin{equation} \label{eq:triangle}
  \tilde{M}(A_q) \to \tilde{M}(Q') \to M(Q)(1)[2] \to \tilde{M}(A_q)[1].
\end{equation}
Write $q = n\hh \perp q_{a}$ with $q_{a}$ anisotropic. Then $q' = n\hh + q_{a}'$ (but $q_a'$ need not be anisotropic).

It follows from triangle \eqref{eq:triangle} and equation \eqref{eq:twist} that we have an exact triangle
$$
\Phi^k(\tilde{M}(A_q)) \to \Phi^k(\tilde{M}(Q_a'))(n)[2n] \to \Phi^k (M(Q_a))(n+1)[2n+2] \to \Phi^k(\tilde{M}(A_q))[1].
$$
Since $Q_a$ is anisotropic, we have $\Phi^k (M(Q_a)) = 0$, and so $\Phi^k(\tilde{M}(A_q))\cong \Phi^k(\tilde{M}(Q_a'))(n)[2n]$.
There are two cases.
\begin{itemize}
\item[I.] $q_a'$ is also anisotropic.  Then $\Phi^k (M(Q_a')) = 0$ as well, so
\[ \Phi^k(\tilde{M}(A_q)) = \zz/2(n)[2n-1]. \]
\item[II.] $q_a'$ is isotropic, so $q_a' = \hh \perp q_b$. Then $q_b$ is anisotropic
(indeed $q_a \perp \la -1 \ra = q_a' \cong q_b \perp \la 1 \ra \perp \la -1 \ra$, so $q_a \cong q_b \perp \la 1 \ra$ by Witt cancellation; and a subform of an anisotropic form is anisotropic).
By \eqref{eq:rost},
\[\Phi^k(\tilde{M}(A_q)) = \Phi^k(\tilde{M}(Q_a'))(n)[2n] = \zz/2(n + \dim(Q'_a))[2n+2\dim(Q'_a)].\]
\end{itemize}

Thus either (i) $\Phi^k(\tilde{M}(A_q)) = \zz/2(n)[2n+1]$ or (ii) $\Phi^k(\tilde{M}(A_q)) = \zz/2(m)[2m]$.
Note that (i) and (ii) are different: for no value of $n, m$ do we have $\zz/2(n)[2n+1] \cong \zz/2(m)[2m]$.
Situation (i) corresponds to case (I), and situation (ii) corresponds to case (II). In case (I) we know that $i_W(q) = n = i_W(q')$.
In case (II) we read off the number $m = i_W(q) + \dim(Q'_a)$.
But we also know the number $\dim(Q) = 2i_W(q) + \dim(Q_a) = 2i_W(q) + \dim(Q'_a) -1$, by the first part.
From this we can solve for $i_W(q)$ (namely $i_W(q) = \dim(Q) - m + 1$). Finally $i_W(q') = i_W(q) + 1$ in this case.
\Qed
\end{prof}

\noindent
{\small {\bf Acknowledgements:} \ \
We are grateful to the Referee for useful remarks which improved the exposition.}

\end{document}